\documentclass[12pt]{article}
\usepackage{}

\usepackage{epsfig}
\usepackage{latexsym}
\usepackage{caption}
\usepackage{amsfonts}
\usepackage{amssymb}
\usepackage{mathrsfs}
\usepackage{amsmath}
\usepackage{enumerate}
\usepackage{graphics}
\usepackage{MnSymbol}
\usepackage{float}
\usepackage{pict2e}
\usepackage{subfigure}

\usepackage{color}

\makeatletter

\newcommand{\Rmnum}[1]{\expandafter\@slowromancap\romannumeral #1@}

\makeatother

\begin{document}

\newtheorem{theorem}{Theorem}
\newtheorem{observation}{Observation}
\newtheorem{corollary}{Corollary}
\newtheorem{algorithm}[theorem]{Algorithm}
\newtheorem{definition}{Definition}
\newtheorem{guess}{Conjecture}
\newtheorem{claim}{Claim}
\newtheorem{problem}[theorem]{Problem}
\newtheorem{question}{Question}
\newtheorem{lemma}{Lemma}
\newtheorem{proposition}{Proposition}
\newtheorem{fact}{Fact}

\makeatletter
  \newcommand\figcaption{\def\@captype{figure}\caption}
  \newcommand\tabcaption{\def\@captype{table}\caption}
\makeatother

\newtheorem{acknowledgement}[theorem]{Acknowledgement}

\newtheorem{axiom}[theorem]{Axiom}
\newtheorem{case}[theorem]{Case}
\newtheorem{conclusion}[theorem]{Conclusion}
\newtheorem{condition}[theorem]{Condition}
\newtheorem{conjecture}[theorem]{Conjecture}
\newtheorem{criterion}[theorem]{Criterion}
\newtheorem{example}[theorem]{Example}
\newtheorem{exercise}[theorem]{Exercise}
\newtheorem{notation}{Notation}
\newtheorem{solution}[theorem]{Solution}
\newtheorem{summary}[theorem]{Summary}

\newenvironment{proof}{\noindent {\bf
Proof.}}{\rule{3mm}{3mm}\par\medskip}
\newcommand{\remark}{\medskip\par\noindent {\bf Remark.~~}}
\newcommand{\pp}{{\it p.}}
\newcommand{\de}{\em}
\newcommand{\mad}{\rm mad}
\newcommand{\qf}{Q({\cal F},s)}
\newcommand{\qff}{Q({\cal F}',s)}
\newcommand{\qfff}{Q({\cal F}'',s)}
\newcommand{\f}{{\cal F}}
\newcommand{\ff}{{\cal F}'}
\newcommand{\fff}{{\cal F}''}
\newcommand{\fs}{{\cal F},s}
\newcommand{\s}{\mathcal{S}}
\newcommand{\G}{\Gamma}
\newcommand{\g}{(G_3, L_{f_3})}
\newcommand{\wrt}{with respect to }
\newcommand {\nk}{ Nim$_{\rm{k}} $  }
\newcommand{\cl}{ {\cal L}}

\newcommand{\q}{\uppercase\expandafter{\romannumeral1}}
\newcommand{\qq}{\uppercase\expandafter{\romannumeral2}}
\newcommand{\qqq}{\uppercase\expandafter{\romannumeral3}}
\newcommand{\qqqq}{\uppercase\expandafter{\romannumeral4}}
\newcommand{\qqqqq}{\uppercase\expandafter{\romannumeral5}}
\newcommand{\qqqqqq}{\uppercase\expandafter{\romannumeral6}}

\newcommand{\qed}{\hfill\rule{0.5em}{0.809em}}

\newcommand{\var}{\vartriangle}

\title{{\large \bf List $4$-colouring of planar graphs}}
%\title{{The properties of the hypercube-like}}

\author{  Xuding Zhu\thanks{Department of Mathematics, Zhejiang Normal University,  China.  E-mail: xudingzhu@gmail.com.  Grant numbers: NSFC 11971438,U20A2068, ZJNSFC LD19A010001.}}

\maketitle

\begin{abstract}
	
	This paper proves the following result: If $G$ is a planar graph and $L$ is a $4$-list assignment of $G$ such that   $|L(x) \cap L(y)| \le 2$ for every edge $xy$, then $G$ is $L$-colourable. This answers a question asked by   Kratochv\'{i}l,   Tuza  and  Voigt in  [Journal  of Graph Theory, 27(1):43--49, 1998].
	
	\noindent {\bf Keywords:}
	planar graph;  lists with separation; list colouring.

\end{abstract}

\section{Introduction}

A {\em  list assignment} of a graph $G$ is a mapping $L$ which assigns to each vertex $v$ of $G$ a set
$L(v)$ of   permissible colours. An {\em $L$-colouring} of $G$ is a proper colouring $f$ of $G$ such that for each vertex $v$ of $G$, $f(v) \in L(v)$. We say $G$ is {\em $L$-colourable} if $G$ has an $L$-colouring. A {\em $k$-list assignment} of $G$ is a list assignment $L$ with $|L(v)| \ge k$ for each vertex $v$.  We say $G$ is {\em $k$-choosable} if $G$ is $L$-colourable for any $k$-list assignment $L$ of $G$. The {\em choice  number} $ch(G)$ of $G$ is the minimum integer $k$ such that
$G$ is $k$-choosable.

It is  known that there are planar graphs $G$ and $4$-list assignments $L$ of $G$ such that $G$ is not $L$-colourable \cite{Voigt1993}. 
A natural direction of research is to put restrictions on the list assignments so that for any planar graph $G$ and any  4-list assignment $L$ of $G$ satisfying the restrictions, $G$ is $L$-colourable. Indeed, 
the Four Colour Theorem can be formulated as such a result: For any planar graph $G$, if $L$ is a 4-list assignment of  $G$ with $L(x)=L(y)$ for any edge $xy$ of $G$, then $G$ is $L$-colourable. 

Are there other natural restrictions for which the corresponding ``list $4$-colouring theorem" is true?

By changing the equality to inequality in the above formulation of the Four Colour Theorem, one may ask the following question:

\medskip

{\em 
	Is it true that for any planar graph $G$, any $4$-list assignment $L$ of  $G$ such that  $L(x) \ne L(y)$ (or equivalently, $|L(x) \cap L(y)| \le 3$),    $G$ is  $L$-colourable?
}
\medskip

The answer is NO. Mirzakhani \cite{Mirzakhani} constructed   a planar graph $G$  and a $4$-list assignment $L$ of $G$ such that 
$|L(x) \cap L(y)| \le 3$, and $G$ is not $L$-colourable.

On the other hand,   Kratochv\'{i}l,   Tuza  and  Voigt \cite{KTV1998} proved that for any planar graph $G$, for any 4-list assignment $L$ of $G$ such that for any edge $xy$, $|L(x) \cap L(y)| \le 1$,  $G$ is $L$-colourable. 
Then they  asked the following question:

\begin{question}\cite{KTV1998}
\label{q1}
Is it true that for any planar graph $G$ and any   4-list assignment $L$ of  $G$ such that $|L(x) \cap L(y)| \le 2$ for every edge $xy$, $G$ is   $L$-colourable?
\end{question}

This question received a lot of  attention \cite{BCD,CLS,DEKO, EKT, FKK,HZ,KL2015,KMS,Skr,Smith, WWY}. Most of the works deal with variations of this problem. There was not much progress on the question itself.
%,  except that  very recently, it is  proved in \cite{Smith} that $G$ is $L$-colourable for any $4$-list assignment $L$ for which $|L(x) \cap L(y)| \le 2$ and   $|L(x) \cap L(y) \cap L(z)| \le 1$ for each triangle $(x,y,z)$.   
In this paper, we answer this question in  affirmative.

\begin{definition}
	\label{def-s-separated}
	Assume $G$ is a graph and $k, s$ are   positive integers. A {\em $(\star,  s)$-list assignment } of $G$ is a  list assignment $L$ of $G$ such that   $|L(x)\cap L(y)| \le s$ for each edge $xy$. A {\em $(k,  s)$-list assignment } of $G$ is a  $(\star,  s)$-list assignment   of $G$ with $|L(v)|\ge k$ for each vertex $v$.
	A graph $G$ is called $(k,  s)$-choosable if $G$ is $L$-colourable for any  $(k,  s)$-list assignment $L$ of $G$.
\end{definition}
The following is the main result of this paper.

\begin{theorem}\label{main}
	Every planar graph is $(4,  2)$-choosable.
\end{theorem}

\section{The proof}

It  suffices to prove Theorem \ref{main} for 2-connected near-triangulations of the plane: for each   non-triangular  face, we can add a new vertex adjacent to all vertices of the face,      and assign new colours to the added vertex so that the resulting list assignment is still a $(4, 2)$-list assignment.  For   a 2-connected plane graph $G$, we denote by $B(G)$ the boundary cycle of $G$.

\begin{definition}
	\label{def-part}
	A  {\em rooted plane graph}  is a  pair $(G, v_1v_2)$, where $G$ is a 2-connected near-triangulation   of the plane, and  
  $v_1v_2$ is a boundary edge. 
\end{definition}
 The  vertices $v_1, v_2$  are called the  {\em root vertices} and $v_1v_2$ is called the {\em root edge}.
\begin{definition}
	Assume  $(G, v_1v_2)$ is a rooted plane graph. Assume $v$ is a non-root boundary vertex, and $N_G(v) \cap B(G) = \{u_1,u_2, \ldots, u_k\}$ ($k \ge 2$), and $(v, u_1, u_2, \ldots, u_k)$ occur in $B(G)$ is this cyclic order. The vertices $u_1, \ldots, u_k$ are called the {\em boundary neighbours} of $v$. If the rooted edge is contained in the boundary path from $u_i$ to $u_{i+1}$, then $u_i$ and $u_{i+1}$ are called the {\em primary boundary neighbours} of $v$. 
\end{definition}

Note that each non-root boundary vertex of a rooted plane graph has exactly two primary boundary neighbours.

\begin{definition}
	A  {\em   list assignment} of a rooted plane graph  $(G, v_1v_2)$ is a mapping $L$ which assigns to each vertex $v \ne v_1,v_2$ a set $L(v)$ of colours, and assigns to the ordered pair $(v_1,v_2)$ a set $L(v_1,v_2)$ of ordered   pairs of distinct colours. An $L$-colouring of $(G, v_1v_2)$ is a proper colouring $f$ of $G$ such that for each $v \ne v_1,v_2$, $f(v) \in L(v)$, and $(f(v_1), f(v_2)) \in L(v_1,v_2)$.
\end{definition}

Assume $L$ is a list assignment of $(G, v_1v_2)$. The list assignment $\tilde{L}$ of $G$ associated to $L$ is the list assignment of $G$ defined as $\tilde{L}(v)=L(v)$ for $v \ne v_1, v_2$ and $\tilde{L}(v_1)=\{c: \exists d, (c,d) \in L(v_1,v_2)\}$ and 
$\tilde{L}(v_2)=\{d: \exists c, (c,d) \in L(v_1,v_2)\}$. We say $L$ is a $(\star,  2)$-list assignment of $(G, v_1v_2)$ if $\tilde{L}$ is a 
$(\star,  2)$-list assignment of $G$. 

\begin{definition}
	Assume $L$ is a  $(\star,  2)$-list assignment of $(G, v_1v_2)$, and $v \in B(G)$ is a non-root vertex, and    
  $u$ is a primary boundary neighbour of $v$. We say $u$ is a {\em good   neighbour} of $v$, if one of the following holds:
  \begin{itemize}
  	\item $|\tilde{L}(u) \cap \tilde{L}(v)| \le 1$, or
  	\item $|\tilde{L}(u)|=4$.
  \end{itemize}    
\end{definition}

\begin{definition} A $(\star,  2)$-list assignment of a rooted plane graph $(G, v_1v_2)$ is  {\em valid}    if    $|L(v)| =4$ for each interior vertex $v$, and  one of the following holds:
	\begin{enumerate}
		\item[(A)]  $|L(v_1,v_2)| \ge 1$ and $|L(v)| \ge 3$ for each non-root boundary vertex $v$.  
		\item[(B)] $|L(v_1,v_2)| \ge 2$, and there exists a unique non-root boundary vertex $v^*$ such that   $|L(v)| \ge  3$ for $v \in B(G)-\{v_1,v_2, v^*\}$,   $|L(v^*)| =2$ and $v^*$ has a   good   neighbour.  
	\end{enumerate}
\end{definition}

 Assume $L$ is a valid list assignment of $(G, v_1v_2)$,  $v^*$ is a non-root boundary vertex with $|L(v^*)|=2$ and  $u$ is a good neighbour of $v^*$. If 
 $|L(u)|=4$, then we may delete one colour from $L(u) \cap L(v^*)$ so that $|L(u) \cap L(v^*)| \le 1$. So if $L$ is a valid list assignment of $(G, v_1v_2)$,   and  $u$ is a good neighbour of $v^*$, then we assume that   $|\tilde{L}(u) \cap L(v)| \le 1$. However to prove that $u$ is a good neighbour of $v^*$, it suffices to prove that either $|L(u)|=4$ or  $|\tilde{L}(u) \cap L(v)| \le 1$.

\begin{theorem}
	\label{thm-main}
 If $L$ is a valid list assignment of a rooted plane graph $(G,v_1v_2)$, then
	there exists an $L$-colouring of  $(G, v_1v_2)$. 
\end{theorem}
\begin{proof}
	The proof is by induction on $|V(G)|$.
		
	Assume first that $G$ is a triangle $(v_1,v_2,v_3)$.

	If (A) holds, then $|L(v_3)|=3$. Assume $  L(v_1,v_2) = \{(c_1,c_2)\}$. Let $c_3 \in L(v_3) - \{c_1,c_2\}$. Then $f(v_i)=c_i$ for $i=1,2,3$   is an $L$-colouring of $(G, v_1v_2)$.  
	
	Assume (B) holds. Then $|L(v_3)|=2$ and $L(v_1,v_2)=\{(c_1,c_2),(c'_1,c'_2)\}$. We may assume that $v_2$ is a good   neighbour of $v_3$. 
	If $c_1=c'_1$, then $c_2 \ne c'_2$. Let $c_3 \in L(v_3)-\{c_1\}$. One of $c_2, c'_2$ is distinct from $v_3$.  Without loss of generality, we may assume that $c_2 \ne c_3$. Then   $f(v_i)=c_i$ for $i=1,2,3$   is an $L$-colouring of $(G, v_1v_2)$.
	The case $c_2=c'_2$ is symmetric.
	
	Assume $c_1 \ne c'_1, c_2 \ne c'_2$. As $v_2$ is a good   neighbour of $v_3$,     $|L(v_3) \cap \{c_2,c'_2\}| \le 1$. Assume $c_2 \notin L(v_3)$. Let $c_3 \in L(v_3) - \{c_1\}$.  Then   $f(v_i)=c_i$ for $i=1,2,3$   is an $L$-colouring of $(G, v_1v_2)$.

	Assume $|V(G)|=n \ge 4$ and the theorem is true for any smaller rooted plane graphs. 
	
	For a 
	cycle $C$ of $G$, ${\rm Int}[C]$ is the graph of all vertices and edges inside or on $C$, ${\rm Ext}[C]$ is the graph of all vertices and edges outside or on $C$. If $G$ has a separating triangle $C=(u_1,u_2,u_3)$, then let $G_1=  {\rm Ext}[C]$. Then $(G_1, v_1v_2)$ has an $L$-colouring $f$. Let $G_2= {\rm Int}[C]-\{u_3\}$.   Let $L'$ be the list assignment of $(G_2, u_1u_2)$ defined as 
	$L'(u_1,u_2)=\{(f(u_1), f(u_2))\}$, and for $v \in V(G_2)-\{u_1,u_2\}$,
	\[
	L'(v)=\begin{cases} L(v)-\{f(u_3)\}, &\text{if $v \in N_G(u_2)$}, \cr
	L(v), &\text{otherwise}. \cr
	\end{cases}
	\]
	Then $L'$ is a valid list assignment of $(G_2, u_1u_2)$. By induction hypothesis, there is an $L'$-colouring $g$ of $(G_2, u_1u_2)$. The union of $f$ and $g$ is an $L$-colouring of $(G, v_1v_2)$. 
	
	In the following, we assume that $G$ has no separating triangle. 
		
		\medskip
		\noindent
		{\bf Case 1}
		$B(G)$ has a chord $xy$. 
		
		Let $G_1,G_2$ be the two subgraphs of $G$ separated by $xy$, (i.e.,  $G_1,G_2$ are connected  induced subgraphs of $G$  with $V(G_1)\cap V(G_2)=\{x,y\}$ and $V(G_1)\cup V(G_2)=V(G)$), and assume $ G_1$   contains the root edge $v_1v_2$. 
		
		\medskip
		\noindent
		{\bf Case 1(i)}	There is a chord $xy$ such  that $|L(v)|=3$ for all $v \in B(G_2) - \{x,y\}$.
		
			Let $L_1$ be the restriction of $L$ to $(G_1, v_1v_2)$. 
		Then  $L_1$ is a valid list assignment of $(G_1,v_1v_2)$.	By induction hypothesis, there exists an $L_1$-colouring $f$ of $(G_1, v_1v_2)$.
		
Let $L_2$ be the list assignment of $(G_2, xy)$ defined as  $L_2(x,y)=\{(f(x),f(y))\}$ and $L_2(v) = L(v)$ for $ v \in V(G_2) - \{x,y\}$. Then $L_2$ is a valid list assignment of $(G_2, xy)$. 
	By induction hypothesis,   there exists an $L_2$-colouring $g$ of $(G_2, xy)$. The union of $f$ and $g$ is an $L$-colouring of $(G, v_1v_2)$. 
	
\medskip
		\noindent
		{\bf Case 1(ii)}	There is a vertex   $v^* \in B(G) - \{v_1,v_2\}$ with $|L(v^*)|=2$, and every chord $xy$ separates $v^*$ and the root edge, i.e., $v_1v_2 \in E(G_1)$ and $v^* \in V(G_2) - \{x,y\}$.

		We choose the chord $xy$ so that
		$G_1$ is minimum. Then $B(G_1)$ has no  chord.

As there is a vertex   $v^* \in B(G) - \{v_1,v_2\}$ with $|L(v^*)|=2$, we know that $(G, v_1v_2)$ satisfies (B). We may assume that $|L_1(v_1,v_2)| = 2$.

Similarly, $L_1$ is a valid list assignment of $(G_1, v_1v_2)$ and hence there is an $L_1$-colouring $f$ of $(G_1, v_1v_2)$.

\begin{claim}
\label{clm-1}
There is another $L$-colouring $f'$ of $(G_1, v_1v_2)$ for which $(f'(x), f'(y)) \ne (f(x), f(y))$.
\end{claim}

  Assume Claim \ref{clm-1} is true. Let $L_2$ be the list assignment of $(G_2, xy)$   defined as  $L_2(x,y)=\{(f(x),f(y)), (f'(x), f'(y))\}$ and $L_2(v) = L(v)$ for $ v \in V(G_2) - \{x,y\}$.  Note that $\tilde{L}_2(v) \subseteq L(v)$ for $v \in \{x,y\}$, and the primary neighbours of $v^*$ in $(G_2, xy)$ are the same as its primary neighbours in $(G,v_1v_2)$. So $v^*$ has a good neighbour in $(G_2, xy)$. Thus $L_2$ is a valid list assignment of $(G_2, xy)$.

	By induction hypothesis,   there exists an $L_2$-colouring $g$ of $(G_2, xy)$. Depending on 
	$(g(x), g(y)) = (f(x), f(y))$ or $(f'(x), f'(y))$, the
	union of   $g$ and $f$ or the union of $g$ and $f'$ is an $L$-colouring of $(G, v_1v_2)$. To finish the proof of Case 1, it remains to prove Claim \ref{clm-1}. 

\medskip	
\noindent
{\bf Proof of Claim \ref{clm-1}}

Without loss of generality, we assume that   $y \notin \{v_1, v_2\}$. 
Let $L'_1  = L_1$, except that $L'_1(y) = L(y)-\{f(y)\}$. If $L'_1$ is a valid list assignment of $(G_1, v_1v_2)$, then by induction hypothesis, there is an $L'_1$-colouring $f'$ of $(G_1, v_1v_2)$, and we are done.

Thus we may assume that $L'_1$ is not a valid list assignment of $(G_1, v_1v_2)$. This happens only if $|L'_1(y)|=2$ and 
$y$ has no good neighbour in $(G_1, v_1v_2)$. Assume $L'_1(y) = \{c_1,c_2\}$.

As $B(G_1)$ has no chord,  $y$ has exactly two boundary neighbours, and one of them is $x$.  Let $y'$ be the other boundary neighbour of $y$, i.e.,  $N_G(y) \cap B(G_1)=\{x, y'\}$. Then   
		$$\{c_1, c_2\} \subseteq \tilde{L}'_1(x) \cap \tilde{L}'_1(y').$$

	If $x$ is a root vertex, say $x=v_1$, then 
	  $\tilde{L}_1(x) = \{c_1,c_2\}$ (as $L'_1(y) \subseteq \tilde{L}_1(x)$ and $\tilde{L}_1(x)| \le 2$). 
	  
	  Assume $L(v_1,v_2)=\{(c_1, c'_1), (c_2, c'_2)\}$ for some colours $c'_1, c'_2$ (possibly $c'_1=c'_2$). Assume $c_1 \ne f(x)$.
		
Let $L''_1 = L_1$, except that $	L''_1(v_1,v_2)=(c_1, c'_1)$. As $|L''_1(v)|\ge 3$ for all $v \in B(G_1)-\{v_1,v_2\}$,   $L''_1$ is a valid list assignment of $(G_1,v_1v_2)$. Hence there is an $L''_1$-colouring $f'$ of $(G_1,v_1v_2)$. As $f'(x) \ne f(x)$,  Claim \ref{clm-1} is proved.

  Thus we may assume that $x$ is not a root vertex.  
		
%Recall that $f$ is an $L_1$-colouring of $(G_1, v_1v_2)$.
		Let $L''_1 =L_1$ except that $L''_1(x) = L_1(x) - \{f(x)\}$. If $L''_1$ is   a valid list assignment of $(G_1, v_1v_2)$, then again we obtain an $L$-colouring $f'$ of $(G_1, v_1v_2)$ with $(f'(x), f'(y)) \ne (f(x), f(y))$ and we are done. 
		
		Thus assume that $L''_1$ is not a valid list assignment. This means that $|L''(x)|=2$ and $x$ has no good neighbour. Let $x'$ be the 
		other   boundary neighbour of $x$. 
		 Then we have $L''_1(x)=\{c_1,c_2\}$ (so $f(x) \ne c_1,c_2$), and 
		$$\{c_1,c_2\} = L(x) \cap L(y) \cap \tilde{L}_1(x') \cap \tilde{L}_1(y').$$
	 
		As $G$ is a near-triangulation of the plane and $G$ has no separating triangle, 
		 $x$ and $y$ have a unique common neighbour $z$ in $G_1$, which  is an interior vertex of $G_1$.
		 
	Since $B(G_1)$ has no chord, it is easy see that at least one of the following holds:
		 	\begin{itemize}
		 		\item $N_{G_1}(x') \cap N_{G_1}(y) - \{z\} = \emptyset$.
		 		\item $N_{G_1}(y') \cap N_{G_1}(x) - \{z\} = \emptyset$.
		 	\end{itemize}
	 By symmetry, we may assume that 
	 $N_{G_1}(x') \cap N_{G_1}(y) - \{z\} = \emptyset$.
	 
	  As $|L(z) \cap L(y)| \le 2$, there exists $i \in \{1,2\}$, that $|L(z) \cap \{f(y), c_i\}| \le 1$. Without loss of generality, we assume  $$|L(z) \cap \{f(y), c_1\}| \le 1.$$
	  
	  Let $$G'_1=G_1-\{x,y\}.$$

	Let $L^*_1$ be the list assignment of $(G'_1, v_1v_2)$ defined as follows:
		\[
		L^*_1(v)=\begin{cases} L(v)-\{c_1, f(y)\}, &\text{if $v =z$}, \cr
		L(v)-\{f(y)\}, &\text{if $v \in N_{G_1}(y)-\{z\}$},\cr
		L(v)-\{c_1\}, &\text{if $v \in N_{G_1}(x)-\{z\}$},\cr
		L(v), &\text{otherwise}.
		\end{cases}
		\]
		and 
		\[
		L^*_1(v_1, v_2)=\begin{cases}
		L(v_1,v_2),
		&\text{  if $x'$ is not a root vertex, or
		$c_1 \notin \tilde{L}(x')$}, \cr 
		   L(v_1,v_2) - \{(c_1, c'_1)\}, &\text{ if $x'=v_1$ and  $  (c_1,c'_1) \in L(v_1,v_2)$}.
		\end{cases}  
		\]

		Note that $|L^*_1(z)| \ge 3$, and $L^*_1(y')= L(y')$ (as $f(y) \notin L(y')$). 
		If $L^*_1$ is  a valid list assignment of $(G'_1, v_1v_2)$, then   there is an $L^*_1$-colouring $f'$ of $(G'_1, v_1v_2)$. 
		By letting $(f'(x), f'(y))=(c_1, f(y))$, we obtain an $L$-colouring of $(G_1, v_1v_2)$ with $(f'(x), f'(y)) \ne (f(x), f(y))$, and we are done. 
		
		Thus we assume that $L^*_1$ is not a valid list assignment of $(G'_1, v_1v_2)$.

	This means that  
	\begin{itemize}
	    \item $x'$ is not a root vertex,   $x'$  is the only  boundary vertex of $G'_1$ with $|L^*_1(x')|=2$, 
	    and $x'$ has no good neighbour.
	\end{itemize} 
	
	Assume $L(x') = \{c_1,c_2,c_3\}$ (and hence $L^*_1(x')=\{c_2, c_3\}$). 
	
	Let $z'$ be the unique common neighbour  of $x$ and $x'$, which is an interior vertex of $G_1$. 
	
		Let $x''$ be the other   neighbour of $x'$ in $B(G_1)$. 
		Then $x''$  is a primary boundary neighbour of $x'$ in $G'_1$. 
		Since $ N_{G_1}(x') \cap N_{G_1}(y) - \{z\} = \emptyset$ and $G$ has no separating triangle, $z'$ is the other primary boundary neighbour of $x'$.

Since 
 $z'$ is not a good neighbour of $x'$, we conclude that  
	\begin{itemize}
	    \item $z'=z$;
	    \item $c_1 \notin L(z), c_2, c_3,  f(y) \in L(z)$ and $c_3 \ne f(y)$.
	\end{itemize}    
	
Now $z'=z$ implies that $N_{G_1}(y') \cap N_{G_1}(x) - \{z\} = \emptyset$.
So we can  repeat the same argument as above, but interchange the roles of $x,x'$ and $y, y'$. Then we  conclude that the following hold:
	\begin{itemize}
	    \item $y'$ is not a root vertex,
	    \item $z$ is adjacent to $y'$,
	    \item $L(y')=\{c_1, c_2, c'_3\}$,
	    \item $c_2,c'_3, f(x) \in L(z)$. 
	\end{itemize}

	As $\{c_2, c_3, c'_3, f(x), f(y) \} \subseteq L(z_1)$, we have   $c_3=c'_3$ (as $|L(z)|=4$ and the other colours are pairwise distinct). I.e., 
	$$L(z) = \{c_2,c_3,f(x),f(y)\}.$$

	As $L^*_1(x')=\{c_2, c_3\} \subseteq \tilde{L}(x'')$,  we know that $c_1 \notin \tilde{L}(x'')$. 
	
		Let $$G''_1=G_1-\{x',x,y\}.$$
		
	Let $L^{**}_1$ be the list assignment of $(G''_1, v_1v_2)$ defined as follows:
		 
		\[
		L^{**}_1(v)=\begin{cases} L(v)-\{c_1, c_2, f(y)\}, &\text{if $v =z$}, \cr
		L(v)-\{f(y)\}, &\text{if $v \in N_{G_1}(y)$},\cr
		L(v)-\{c_2\}, &\text{if $v \in N_{G_1}(x)$},\cr
		L(v)-\{c_1\}, &\text{if $v \in N_{G_1}(x')$}, \cr
		L(v), &\text{otherwise}.
		\end{cases}
		\]
		and 
		\[
		L^{**}_1(v_1, v_2)=\begin{cases}
		L(v_1,v_2),
		&\text{  if $x''$ is not a root vertex, or
		$c_1 \notin \tilde{L}(x'')$}, \cr 
		   L(v_1,v_2) - \{(c_1, c'_1)\}, &\text{ if $x''=v_1$ and  $  (c_1,c'_1) \in L(v_1,v_2)$}.
		\end{cases}  
		\]
If $L^{**}_1$ is a valid list assignment of $(G''_1, v_1v_2)$, then there is an $L^{**}_1$-colouring $f'$ of $(G''_1, v_1v_2)$,
which extends to an $L$-colouring of $(G_1,v_1v_2)$ by letting $f'(x')=c_1, f'(x)=c_2$ and $f'(y)=f(y)$, and we are done. 

Thus we may assume that $L^{**}_1$ is not a valid list assignment of $(G''_1, v_1v_2)$.   It is easy to check that $z$ is the only vertex of $B(G''_1)-\{v_1, v_2\}$ with $|L^{**}(z)| < 3$. Note that $$L^{**}(z) = L(z) - \{c_2, f(y)\} = \{c_3,f(x)\}.$$
The only reason that $L^{**}_1$ is not a valid list assignment of $(G''_1, v_1v_2)$ is that $z$ has no good neighbour. Let $w_1,w_2$ be the two primary boundary neighbours of $z$ in $(G''_1, v_1v_2)$. We have 
$$\{c_3, f(x)\} \subseteq \tilde{L}(w_1), \tilde{L}(w_2).$$
This implies that $y'$ is not a primary boundary neighbour of $z$ in $(G''_1, v_1v_2) $ (although $y'$ is a boundary neighbour of $z$ in $G''_1$). 

We repeat the above argument, but interchange the roles of $x,x'$ and $y,y'$. We conclude that for  the two primary boundary neighbours $w'_1,w'_2$ of $z$ in $(G_1-\{x,y,y'\}, v_1v_2)$, 
		$$\{c_3, f(y)\} \subseteq \tilde{L}(w'_1), \tilde{L}(w'_2).$$
This means that $x'$ is not a primary neighbour of $z$ in $(G_1-\{x,y,y'\}, v_1v_2)$. But then the primary neighbours of $z$ in $(G_1-\{x,y,y'\}, v_1v_2)$ and $(G_1-\{x',x,y\}, v_1v_2)$ are the same. I.e., $w'_1=w_1$ and $w'_2=w_2$. But then for $i=1,2$,
		$$\{c_3, f(x),f(y)\} \subseteq \tilde{L}(w_i) \cap L(z),$$
		contrary to the assumption that $L$ is $(\star,  2)$-list assignment of $(G, v_1v_2)$.  
		This completes the proof of Claim \ref{clm-1}, and hence the proof of Case 1.
		
	\medskip
	\noindent
	{\bf Case 2}
	$B(G)$ has  no chord. 
		\medskip
		
	\medskip
	\noindent
	{\bf Case 2(i)}	  (A) holds, and $L(v_1,v_2)=\{(c_1,c_2)\}$. 
	
	Let   $u$ be the other boundary neighbour of $v_2$ in $G$.
Similarly, as $G$ has no separating triangle and $B(G)$ has no chord,   $v_1$ and $v_2$ has a unique common neighbour $w$, and $u$ and $v_2$
  have a unique common neighbour $z$, and $w,z$ are interior vertices of $G$ (and possibly $w=z$).
  
  Let $G'=G-v_2$ and 
	let $L'$ be the list assignment of $(G', v_1w)$ defined as
	\[
	L'(v) = \begin{cases}  
	L(v)-\{c_2\}, & \text{ if $v \in N_G(v_2) - \{v_1, w\}$}, \cr
	L(v), &\text{ if $v \in V(G)-N_G(v_2)$},
	\end{cases}   
	\] 
	and  
	\[
	L'(v_1,w)=\{(c_1, c_3), (c_1, c_4)\}, \text{where $c_3, c_4 \in L(w) -\{c_1,c_2\}$}.
	 \]
	
	In the definition above, if $|L(w) -\{c_1,c_2\}| \ge 3$, then $c_3, c_4$ are   arbitrarily chosen  from $L(w) -\{c_1,c_2\}$, with one exception:

	If $c_2 \notin L(w)$,    $w=z$ and $L(w) \cap L(u) \ne \emptyset$, then let $c' \in L(w) \cap L(u)$, and we choose
	$c_3, c_4 \in L(w) -\{c_1,c'\}$.

	 We shall show that $L'$ is valid list assignment of $(G', v_1w)$.

	If $c_2 \notin L(u)$, then $|L'(u)|=|L(u)|=3$, and (A) holds for $L'$ and $(G', v_1w)$. So $L'$ is a valid list assignment of $(G', v_1w)$. 
	
	Assume $c_2 \in L(u)$ and hence $|L'(u)|=2$. If $z \ne w$, then   either $c_2 \in L(z)$ and hence $|L'(z) \cap L'(u)| \le 1$
	or $|L'(z)|=4$. So   $z$ is a good   neighbour of $u$ in $(G', v_1w)$, and $L'$ is a valid list assignment of $(G', v_1w)$ ((B) holds for $L'$ and $(G', v_1w)$). 
	
	If $z = w$, then by our choice of $c_3, c_4$, we know that    $|\tilde{L}'(w) \cap L'(u)|  \le 1$,
and hence $w$ is a good neighbour of $u$, and  $L'$ is a valid 
	list assignment of $(G', v_1w)$ ((B) holds for $L'$ and $(G', v_1w)$).

	By induction hypothesis,    $(G', v_1w)$  has an $L'$-colouring $f$. By letting $f(v_2)=c_2$, we obtain an $L$-colouring of $(G, v_1v_2)$.  
	
\medskip
\noindent
{\bf Case 2(i)}	 (B) holds, and  $v^* \in B(G)$, $|L(v^*)|=2$, and $u$ is a good   neighbour of $v^*$. 
	
	It may happen that $v^*$ has two good neighbours. In this case, the good neighbour $u$ is usually arbitrarily chosen, unless $v^*$ is adjacent to a root vertex $v_i$ for some $i \in \{1,2\}$ and $|\tilde{L}(v_i)|=1$. In this case, we let $u=v_i$.

	 Let $w$ be the other boundary neighbour of $v^*$, and  let $z$ be the common neighbours of $v^*$ and $w$. Similarly, we know that the vertex $z$ is unique and is an interior vertex of $G$.
	
	 By our choice of $u$, we know that 
  either $w \ne v_1, v_2$, or $w=v_i$ for some $i \in \{1,2\}$ and $|\tilde{L}(v_i)|=2$  (for otherwise, we would have chosen $w$ as the   good neighbour of $v^*$).  
	
	Let $$G'=G-\{v^*\}, \ c \in L(v^*) - L(u).$$
	
	If $w$ is not a root vertex, then let $L'$ be the list assignment of $(G', v_1v_2)$ defined as $L'(v_1,v_2)=L(v_1,v_2)$, and for $v \in V(G')-\{v_1,v_2\}$,
	\[
	L'(v) = \begin{cases}  
	L(v)-\{c\}, & \text{ if $v \in N_G(v^*)$}, \cr
	L(v), &\text{ if $v \in V(G)-N_G(v^*)$}.
	\end{cases}   
	\]  
	If $|L'(w)| \ge 3$, then $|L'(v)| \ge 3$ for every $v \in B(G')-\{v_1,v_2\}$ and hence $L'$ is a valid list assignment of $(G', v_1v_2)$. Otherwise,
  $w$ is the unique boundary vertex of $G'$ with $|L'(w)|=2$.  
	Observe that either $c \in L(z)$ and hence $|L'(z) \cap L'(w)| \le 1$, or $|L'(z_1)|=|L(z_1)|=4$. In any case, $z$ is a good neighbour of $w$, and hence $L'$ is a valid list assignment of $(G, v_1v_2)$.   By induction hypothesis, there is an $L'$-colouring $f$ of $(G', v_1v_2)$. By letting $f(v^*)=c$, we obtain an $L$-colouring of $(G', v_1v_2)$. 
	
	Assume $w$ is a root vertex, say $w=v_1$. If $c \notin \tilde{L}(v_1)$, then the argument still works. Assume $c \in \tilde{L}(v_1)$. Without loss of generality, we may assume that $|L(v_1, v_2)|=2$, say $L(v_1, v_2)=\{(c, d), (c', d')\}$. 
	As observed above, $|\tilde{L}(v_1)|=2$, i.e., $c \ne c'$ (and it is possible that $d=d'$). 
	Let $L'$ be the list assignment of $(G', v_1v_2)$ defined as $L'(v_1, v_2)=\{(c',d')\}$ and 
		\[
		L'(v) = \begin{cases}  
		L(v)-\{c\}, & \text{ if $v \in N_G(v^*)$}, \cr
		L(v), &\text{ if $v \in V(G)-N_G(v^*)$}.
		\end{cases}   
		\]  
		Then for all $v \in B(G')-\{v_1, v_2\}$, $|L'(v)| \ge 3$. Hence $L'$ is a valid list assignment of $(G', v_1v_2)$. By induction hypothesis, there is an $L'$-colouring $f$ of $(G', v_1v_2)$. By letting $f(v^*)=c$, we obtain an $L$-colouring of $(G, v_1v_2)$. 
		
		This completes the proof of Theorem \ref{thm-main}.
	\end{proof}
	
	It is obvious that Theorem \ref{main} follows from Theorem \ref{thm-main}.

\section{Some Remarks and Questions}

For list colouring of planar graphs with list of separation, the following conjecture was propose in   \cite{Skr} and remains open:

\begin{conjecture}
Every planar graph is $(3,  1)$-choosable.
\end{conjecture}

There are some other restrictions on list assignments are studied in the literature \cite{CK2017,KR2002,Zhu}. 
We say a list assignment $L$ is {\em symmetric} if colours in the lists are integers and for each $v$, for each integer $i$, $i \in L(v)$ implies that $-i \in L(v)$. A graph $G$ is called {\em weakly $k$-choosable} if $G$ is $L$-colourable for any symmetric $k$-list assignment $L$ of $G$.
The following  conjecture, which is a strengthening of the Four Colour Theorem,  was proposed by     K\"{u}ndgen and Ramamurthi \cite{KR2002} and remains open.

\begin{conjecture}
Every planar graph is weakly $4$-choosable.  
\end{conjecture}

A $t$-common $k$-list assignment of a graph $G$ is a $k$-list assignment $L$ of $G$ such that $|\bigcap_{v \in V(G)} L(v)| \ge t$. It was asked by Choi and Kwon \cite{CK2017} whether every planar graph $G$ is $L$-colourable for any 2-common 4-list assignment $L$. A positive answer would be a strengthening of the Four Colour Theorem. But Kemnitz and Voigt \cite{KV} proved that the answer to this question is negative.

\end{document}